\documentclass[11pt,a4paper,twoside]{amsart}

\usepackage{epsfig}
\usepackage{amsfonts}
\usepackage{amsmath}
\usepackage{amssymb}
\usepackage{amsthm}
\usepackage{enumerate}
\usepackage[T1]{fontenc}
\usepackage[latin9]{inputenc}
\usepackage{ae,aecompl}
\usepackage{graphicx}
\usepackage{psfrag}
\usepackage{nicefrac}
\usepackage{color}

\usepackage{hyperref}

\newtheorem{theorem}{Theorem}[section]
\newtheorem{corollary}[theorem]{Corollary}
\newtheorem{lemma}[theorem]{Lemma}
\newtheorem{proposition}[theorem]{Proposition}
\newtheorem{conjecture}[theorem]{Conjecture}
{\theoremstyle{remark}
\newtheorem{remark}[theorem]{Remark}}
{\theoremstyle{definition}
\newtheorem{definition}[theorem]{Definition}

\newtheorem{example}[theorem]{Example}

}

\newcommand{\PP}[0]{\ensuremath{\mathbb{P}}}

\newcommand{\ZZ}[0]{\ensuremath{\mathbb{Z}}}
\newcommand{\GA}[0]{\ensuremath{\mathbb{G}_{\mathrm{a}}}}
\newcommand{\AF}[0]{\ensuremath{\mathbb{A}}}

\newcommand{\RR}[0]{\ensuremath{\mathbb{R}}}
\newcommand{\QQ}[0]{\ensuremath{\mathbb{Q}}}
\newcommand{\TT}[0]{\ensuremath{\mathbb{T}}}
\newcommand{\KK}[0]{\ensuremath{\mathbf{k}}}
\newcommand{\OO}[0]{\ensuremath{\mathcal{O}}}

\newcommand{\DD}[0]{\ensuremath{\mathfrak{D}}}
\newcommand{\tA}[0]{\ensuremath{\widetilde{A}}}
\newcommand{\tX}[0]{{\ensuremath{\widetilde{X}}}}

\newcommand{\fract}[0]{\ensuremath{\operatorname{Frac}}}
\newcommand{\LND}[0]{\ensuremath{\operatorname{LND}}}

\newcommand{\spec}[0]{\ensuremath{\operatorname{Spec}}}
\DeclareMathOperator{\SPEC}{\mathbf{Spec}}
\newcommand{\supp}[0]{\ensuremath{\operatorname{Supp}}}

\newcommand{\ML}[0]{\ensuremath{\operatorname{ML}}}
\newcommand{\FML}[0]{\ensuremath{\operatorname{FML}}}

\newcommand{\divi}[0]{\ensuremath{\operatorname{div}}}

\newcommand{\ord}[0]{\ensuremath{\operatorname{ord}}}
\newcommand{\pol}[0]{\ensuremath{\operatorname{Pol}}}

\newcommand{\rank}[0]{\ensuremath{\operatorname{rank}}}

\newcommand{\homo}[0]{\ensuremath{\operatorname{Hom}}}
\newcommand{\relint}[0]{\ensuremath{\operatorname{rel.int}}}
\newcommand{\trdeg}[0]{\ensuremath{\operatorname{tr.deg}}}
\newcommand{\Sym}[0]{\ensuremath{\operatorname{Sym}}}
\newcommand{\Div}[0]{\ensuremath{\operatorname{Div}}}

\begin{document}

\title[$\GA$-actions of fiber type on affine $\TT$-varieties]{$\GA$-actions of fiber type on affine $\TT$-varieties}

\author{Alvaro Liendo}
\address{Universit\'e
Grenoble I, Institut Fourier, UMR 5582 CNRS-UJF, BP 74, 38402
St.\ Martin d'H\`eres c\'edex, France}
\email{alvaro.liendo@ujf-grenoble.fr}

\date{\today}

\thanks{
\mbox{\hspace{11pt}}{\it 2000 Mathematics Subject
Classification}:
14R05, 14R20, 14J50.\\
\mbox{\hspace{11pt}}{\it Key words}: Torus action, $\GA$-action, locally nilpotent derivations,
affine varieties, Makar-Limanov invariant}

\begin{abstract}
Let $X$ be a normal affine $\TT$-variety, where $\TT$ stands for the algebraic torus. We classify $\GA$-actions on $X$ arising from homogeneous locally nilpotent derivations of fiber type. We deduce that any variety with trivial Makar-Limanov (ML) invariant is birationally decomposable as $Y\times\PP^2$, for some $Y$. Conversely, given a variety $Y$, there exists an affine variety $X$ with trivial ML invariant birational to $Y\times\PP^2$. 

Finally, we introduce a new version of the ML invariant, called the FML invariant. According to our conjecture, the triviality of the FML invariant implies rationality. This conjecture holds in dimension at most 3.
\end{abstract}

\maketitle

%\setcounter{tocdepth}{3}
%\tableofcontents

\section*{Introduction}

The main result of the paper consists in a birational characterization of normal affine algebraic varieties with trivial Makar-Limanov invariant. Let us introduce the necessary notation and definitions.

We let $\KK$ be an algebraically closed field of characteristic 0, $M$ be a lattice of rank $n$, and $\TT=\spec\KK[M]\simeq(\KK^*)^n$ be the algebraic torus over $\KK$. A $\TT$-variety $X$ is a variety endowed with an algebraic action of $\TT$. For an affine variety $X=\spec A$, introducing a $\TT$-action on $X$ is the same as to equipe $A$ with an $M$-grading. There are well known combinatorial descriptions of normal $\TT$-varieties. We send the reader to \cite{Dem70} and \cite[Ch. 1]{KKMS73} for the case of toric varieties, to \cite[Ch. 2 and 4]{KKMS73} and \cite{Tim08} for the complexity 1 case i.e., $\dim X=\dim \TT+1$, and to \cite{AlHa06,AHS08} for the general case.

We let $N=\homo(M,\ZZ)$, and $N_{\QQ}=N\otimes\QQ$. Any affine toric variety can be described via a polyhedral cone $\sigma\subseteq N_{\QQ}$. Similarly, the description of normal affine $\TT$-varieties due to Altmann and Hausen \cite{AlHa06} deals with a polyhedral cone $\sigma\subseteq N_{\QQ}$, a normal variety $Y$, and a divisor $\DD$ on $Y$ whose coefficients are polyhedra in $N_{\QQ}$ invariant by translation in $\sigma$.

To introduce a $\GA$-action on an affine variety $X$ is equivalent to fix a locally nilpotent derivation (LND) on its structure ring $A$ \cite[\S 1.5]{Fre06}. Any LND on $A$ can be extended to a derivation on $K=\fract A$ by the Leibniz rule. If an LND of $A$ is homogeneous with respect to the $M$-grading on $A$ we say that the associated $\GA$-action on $X$ is \emph{compatible} with the $\TT$-action. Furthermore, we say that a homogeneous LND $\partial$ (or, equivalently, the associated $\GA$-action) is of \emph{fiber type} if $\partial(K^\TT)=0$ and of \emph{horizontal type} otherwise \cite{FlZa05b,Lie08}.

In \cite{FlZa05b} Flenner and Zaidenberg gave a classification of compatible $\GA$-actions on normal affine $\KK^*$-surfaces. Generalizing this construction, in \cite{Lie08} a classification of $\GA$-actions on normal affine $\TT$-varieties of complexity 1 was given. In Theorem \ref{fiber} below, we extend this classification to $\GA$-actions of fiber type on normal affine $\TT$-varieties of arbitrary complexity.

The Makar-Limanov (ML) invariant \cite{KaMa97} showed to be an important tool for affine geometry. In particular, it allows to distinguish certain varieties from the affine space. For an algebra $A$, this invariant is defined as the intersection of the kernels of all locally nilpotent derivations on $A$. Nevertheless, this invariant is far form being optimal. Indeed, many non-rational varieties share the same ML invariant as that of the affine space \cite{Lie08}. The latter invariant is trivial i.e., $\ML(\AF^n)=\KK$. In Theorem \ref{birML} we give a birational characterization of normal affine varieties with trivial ML invariant.

To avoid such a pathology, we introduce the FML invariant which is defined as as the intersection of the fields of fractions of the kernel of all locally nilpotent derivations on $A$. For an affine variety $X$, we conjecture that $\FML(X)=\KK$ implies that $X$ is rational. In Theorem \ref{FML3} we prove this conjecture in dimension at most 3.

The content of the paper is as follows. In Section 1 we recall some generalities on $\TT$-actions and $\GA$-actions. In Section 2 we obtain our classification of LNDs of fiber type. In Section 3 we introduce the homogeneous ML invariant and show some of its limitations. In Section 4 we establish our principal result concerning the birational characterization. Finally, in Section 5 we introduce and study the FML invariant.

In the entire paper $\KK$ is an algebraically closed field of characteristic 0, except in Section 1.3, where $\KK$ can be non algebraically closed.

The author is grateful to Mikhail Zaidenberg for posing the problem and permanent encouragement, and to Pierre-Marie Poloni for useful discussions.

\section{Preliminaries}
\label{pre}

\subsection{Combinatorial description of $\TT$-varieties}
\label{comb-des}

Let $N$ be a lattice of rank $n$ and $M=\homo(N,\ZZ)$ be its dual lattice. We also let $N_{\QQ}=N\otimes\QQ$, $M_{\QQ}=M\otimes\QQ$, and we consider the natural duality pairing $M_{\QQ}\times N_{\QQ}\rightarrow \QQ$, $(m,p)\mapsto \langle m,p\rangle$. 

Let $\TT=\spec\KK[M]$ be the $n$-dimensional algebraic torus associated to $M$ and
let $X=\spec\,A$ be an affine $\TT$-variety. It is well known that the comorphism $A\rightarrow A\otimes \KK[M]$ induces an $M$-grading on $A$ and, conversely, every $M$-grading on $A$ arises in this way. Furthermore, a $\TT$-action is effective if an only if the corresponding $M$-grading is effective.

In \cite{AlHa06}, a combinatorial description of affine $\TT$-varieties is given. In what follows we recall the main features of this description. Let $\sigma$ be a pointed polyhedral cone in $N_{\QQ}$. We define $\pol_{\sigma}(N_{\QQ})$ to be the set of all polyhedra in $N_{\QQ}$ which can be decomposed as the Minkowski sum of a compact polyhedron and $\sigma$. 

To any polyhedron $\Delta\in\pol_{\sigma}(N_{\QQ})$ we associate its support function $h_{\Delta}:\sigma^\vee\rightarrow \QQ$ defined by $h_{\Delta}(m)=\min\langle m,\Delta\rangle$. Clearly the support function $h_{\Delta}$ is  piecewise linear. Furthermore, $h_{\Delta}$ is concave and positively homogeneous, i.e. 
$$h_{\Delta}(m+m')\geq h_{\Delta}(m)+h_{\Delta}(m'),\ \mbox{and}\ h_{\Delta}(\lambda m)=\lambda h_{\Delta}(m),\forall m,m'\in \sigma^{\vee},\ \forall\lambda\in \QQ_{\geq 0}\,.$$ 

\begin{definition} \label{ppd}
A variety $Y$ is called semiprojective if it is projective over an affine variety. A \emph{$\sigma$-polyhedral divisor} on $Y$ is a formal sum $\DD=\sum_{H}\Delta_H\cdot H$, where $H$ runs over prime divisors on $Y$, $\Delta_H\in\pol_{\sigma}(N_{\QQ})$, and $\Delta_H=\sigma$ for all but finitely many values of $H$. For $m\in\sigma^{\vee}$ we can evaluate $\DD$ in $m$ by letting $\DD(m)$ be the $\QQ$-divisor
$$\DD(m)=\sum_{H\in Y} h_H(m)\cdot H\,.$$
where $h_H=h_{\Delta_H}$. A $\sigma$-polyhedral divisor $\DD$ is called \emph{proper} if the following hold
\begin{enumerate}[(i)]
\item $\DD(m)$ is semiample\footnote{Recall that a divisor $D$ is semiample if $\OO_Y(rD)$ is globally generated for some $r>0$.} and $\QQ$-Cartier for all $m\in\sigma_M^\vee$, and
\item $\DD(m)$ is big for all $m\in\relint(\sigma^\vee)$.
\end{enumerate}
\end{definition}

The following theorem gives a combinatorial description of $\TT$-varieties analogous to the classical combinatorial description of toric varieties.

\begin{theorem}[\cite{AlHa06}] \label{AH}
To any proper $\sigma$-polyhedral divisor $\DD$ on a semiprojective variety $Y$ one can associate a normal finitely generated effectively $M$-graded domain of dimension $\rank M+\dim Y$ given by\footnote{For a $\QQ$-divisor $D$, we let $\OO_Y(D)=\OO_Y(\lfloor D\rfloor)$, where $\lfloor D\rfloor$ is the integral part of $D$.}
$$A[Y,\DD]=\bigoplus_{m\in\sigma^{\vee}_M} A_m\chi^m,\quad \mbox{where}\quad A_m=H^0(Y,\OO_Y(\DD(m))\subseteq K(Y)\,.$$

Conversely, any normal finitely generated effectively $M$-graded domain is isomorphic to $A[Y,\DD]$ for some semiprojective variety $Y$ and some proper $\sigma$-polyhedral divisor $\DD$ on $Y$. 
\end{theorem}

\subsection{Locally nilpotent derivations and $\GA$-actions}

Let $X=\spec\,A$ be an affine variety. A derivation on $A$ is called \emph{locally nilpotent} (LND for short) if for every  $a\in A$ there exists $n\in\ZZ_{\geq 0}$ such that $\partial^n(a)=0$.  Given an LND $\partial$ on $A$, the map $\phi_\partial:\GA\times A\rightarrow A$, $\phi_\partial(t,f)=e^{t\partial}f$ defines a $\GA$-action on $X$, and any $\GA$-action arises in this way. 

In the following lemma we collect some well known facts about LNDs over a field of characteristic 0 (not necessarily algebraically closed), see e.g., \cite{Fre06}.
\begin{lemma} \label{LND}
Let $A$ be a finitely generated normal domain over a field of characteristic 0. For any two LNDs $\partial$ and $\partial'$ on $A$, the following hold.
\begin{enumerate}[(i)]
\item $\ker\partial$ is a normal subdomain of codimension 1.
\item $\ker\partial$ is factorially closed i.e., $ab\in\ker\partial\Rightarrow a,b\in\ker\partial$.
\item If $a\in A$ is invertible, then  $a\in\ker\partial$.
\item If $\ker\partial=\ker\partial'$, then there exist $a,a'\in\ker\partial$ such that $a\partial=a'\partial'$.
\item If $a\in\ker\partial$, then $a\partial$ is again an LND.
\item If $\partial(a)\in (a)$ for some $a\in A$, then $a\in\ker\partial$.
\item  The field extension $\fract(\ker\partial)\subseteq \fract A$ is purely transcendental of degree 1.
\end{enumerate}
\end{lemma}

\begin{definition} \label{LND-equiv}
We say that two LNDs $\partial$ and $\partial'$ on $A$ are \emph{equivalent} if $\ker\partial=\ker\partial'$.
\end{definition}

Let $\DD$ be a proper $\sigma$-polyhedral divisor on a semiprojective variety $Y$, and let $A=A[Y,\DD]$ be the corresponding $M$-graded domain. A derivation $\partial$ on $A$ is called \emph{homogeneous} if it sends homogeneous elements into homogeneous elements. Given a homogeneous LND $\partial$, we define its degree as $\deg\partial=\deg \partial(f)-\deg f$ for any homogeneous $f\in A\setminus\ker\partial$. 

Let $K_Y=K(Y)$. A homogeneous LND $\partial$ on $A$ extends to a derivation on $\fract A=K_Y(M)$, where $K_Y(M)$ is the field of fractions of $K_Y[M]$. The LND $\partial$ is said to be \emph{of fiber type} if $\partial(K_Y)=0$ and \emph{of horizontal type} otherwise. Let $X=\spec\,A$. Geometrically speaking, $\partial$ is of fiber type if and only if the general orbits of the corresponding $\GA$-action on $X$ are contained in the closures of general orbits of the $\TT$-action given by the $M$-grading.

\subsection{Locally nilpotent derivations on toric varieties} \label{LND-toric}

In this section we recall the classification of homogeneous LND given in \cite{Lie08} for toric varieties defined over a field $\KK$ of characteristic 0 (not necessarily algebraically closed).

Given dual bases $\{\nu_1,\cdots,\nu_n\}$ and $\{\mu_1,\cdots,\mu_n\}$ for $N$ and $M$, respectively, we define the partial derivative $\partial_{\nu_i}$ with respect to $\nu_i$ as the homogeneous derivation on $\KK[M]$ given by $\partial_{\nu_i}(\chi^{\mu_j})=\langle\mu_j,\nu_i\rangle=\delta_{ij}$. 

For a pointed polyhedral cone $\sigma$ in the vector space $N_{\QQ}$, we let 
$$A=\KK[\sigma^{\vee}_M]=\bigoplus_{m\in\sigma^{\vee}_M}\KK\chi^m$$
be the affine semigroup algebra of the corresponding affine toric variety $X_\sigma=\spec\, A$. 

If $\sigma=\{0\}$, then $A$ is spanned by the characters which are invertible functions. By Lemma \ref{LND} (iii) any LND on $A$ is trivial. In the following, we fix a ray $\rho$ of $\sigma$, and we let $\tau$ be the codimension 1 face of $\sigma^\vee$ dual to $\rho$. Furthermore, letting $\rho_0\in M$ be the primitive vector of $\rho$ we consider $\mu\in M$ such that $\langle\mu,\rho_0\rangle=1$ and  $H=\rho_0^{\bot}\subseteq M_{\QQ}$.

\begin{definition} \label{srho}
We define
$$S_{\rho}=\sigma_1^\vee\cap(H-\mu)\cap M\,,$$
where $\sigma_1$ is the cone spanned by the rays of $\sigma$ except $\rho$. We have $S_\rho\neq \emptyset$. Furthermore, $e+m\in S_\rho$ whenever $e\in S_{\rho}$ and $m\in\tau_M$.
\end{definition}

The following theorem gives a classification of the homogeneous LND of fiber type on $A$.

\begin{theorem} \label{toric}
To any pair $(\rho,e)$, where $\rho$ is a ray of $\sigma$ and $e\in S_\rho$, we can associate a homogeneous LND $\partial_{\rho,e}$ on $A=\KK[\sigma_M^\vee]$ of degree $e$ with kernel $\ker\partial_{\rho,e}=\KK[\tau_M]$. Conversely, If $\partial\neq 0$ is a homogeneous LND on $A$, then  $\partial=\lambda\partial_{\rho,e}$  for some ray $\rho\subseteq\sigma$, some lattice vector $e\in S_\rho$, and some $\lambda\in\KK^*$.
\end{theorem}
\begin{proof}
The first assertion is Lemma 2.6 in \cite{Lie08} and the second one follows from Theorem 2.7 in \emph{loc. cit.}
\end{proof}

\section{Locally nilpotent derivations of fiber type}
\label{sec-fib}

In this section we give a complete classification of homogeneous LNDs on $\TT$-varieties over an algebraically closed field $\KK$ of characteristic 0. The particular case of complexity 1 is done in \cite[Section 3.1]{Lie08}.

We fix a smooth semiprojective variety $Y$ and a proper $\sigma$-polyhedral divisor 
$$\DD=\sum_{H} \Delta_H\cdot H \quad \mbox{on}\quad Y\,.$$
Letting $K_Y$ be the field of rational functions on $Y$, we consider the affine variety $X=\spec\,A$, where
$$A=A[Y,\DD]=\bigoplus_{m\in\sigma^{\vee}_M}A_m\chi^m,\quad\mbox{with}\quad A_m=H^0\left(Y,\OO(\DD(m))\right)\subseteq K_Y\,.$$

We denote by $h_H$ the support function of $\Delta_H$ so that $\DD(m)=\sum_{H\in Y} h_H(m)\cdot H$. We also fix a homogeneous LND $\partial$ of fiber type on $A$.

We let $\bar{A}=K_Y[\sigma^\vee_M]$ be the affine semigroup algebra over $K_Y$ with cone $\sigma\in N_{\QQ}$. By Lemma Lemma 1.13 in \cite{Lie08} $\partial$ can be extended to a homogeneous locally nilpotent $K_Y$-derivation $\bar{\partial}$ on $\bar{A}$. 

If $\sigma$ has no ray i.e., $\sigma=\{0\}$, then $\bar{\partial}=0$ by Theorem \ref{toric} and so $\partial$ is trivial. In the sequel we assume that $\sigma$ has at least one ray, say $\rho$. Let $\tau$ be its dual codimension 1 face, and let $S_\rho$ be as defined in Lemma \ref{srho}. 

\begin{definition}
For any $e\in S_\rho$, we let $D_e$ be the $\QQ$-divisor on $Y$ defined by\footnote{cf. Lemma 3.3 in \cite{Lie08}}
$$D_e:=\sum_{H}\max_{m\in\sigma^\vee_M\setminus\tau_M}(h_H(m)-h_H(m+e))\cdot H\,.$$
\end{definition}

\begin{remark} \label{Degood}
An alternative description of $D_e$ is as follows. Since the function $h_H$ is concave and piecewise linear on $\sigma^\vee$, the above maximum is achieved by one of the linear pieces of $h_H$ i.e., by one of the maximal cones in the normal quasifan $\Lambda(h_H)$.

For every prime divisor $H$ on $Y$, we let $\{\delta_{1,H},\cdots,\delta_{\ell_H,H}\}$ be the set of all maximal cones in $\Lambda(h_H)$ and $g_{r,H},\ r\in\{1,\cdots,\ell_H\}$ be the linear extension of $h_H|_{\delta_{r,H}}$ to $M_{\QQ}$. Since the maximum is achieved on one of the linear pieces we have
$$\max_{m\in\sigma^\vee_M\setminus\tau_M}(h_H(m)-h_H(m+e))=\max_{r\in\{1,\cdots,\ell_H\}}(-g_{r,H}(e))=-\min_{r\in\{1,\cdots,\ell_H\}}g_{r,H}(e)\,.$$

Since $\tau$ is a codimension 1 face of $\sigma^\vee$, it is contained as a face in one and only one maximal cone $\delta_{r,H}$. We may assume that $\tau\subseteq\delta_{1,H}$. By the concavity of $h_H$ we have $\ g_{1,H}(e)\leq g_{r,H}(e)$ $\forall r$ and so 
$$D_e=-\sum_{H} g_{1,H}(e)\cdot H\,.$$
\end{remark}

We need the following lemma.

\begin{lemma} \label{phie}
For any $e\in S_\rho$ we define $\Phi_e=H^0(Y,\OO_Y(-D_e))$. If $\varphi\in K_Y$ then $\varphi\in\Phi_e$ if and only if $\varphi A_m\subseteq A_{m+e}$ for any $m\in\sigma^\vee_M\setminus\tau_M$.
\end{lemma}
\begin{proof}
If $\varphi\in\Phi_e$, then for every $m\in\sigma^\vee_M\setminus\tau_M$,
$$\divi(\varphi)\geq  D_e\geq \sum_{H}(h_z(m)-h_z(m+e))\cdot H=\DD(m)-\DD(m+e)\,.$$
If $f\in \varphi A_m$ then $\divi(f)+\DD(m)\geq \divi(\varphi)$ and so $\divi(f)+\DD(m+e)\geq 0$. Thus $\varphi A_m\subseteq A_{m+e}$.

To prove the converse, we let $\varphi\in K_Y$ be such that $\varphi A_m\subseteq A_{m+e}$ for any $m\in\sigma^\vee_M\setminus\tau_M$. With the notation of Remark \ref{Degood}, we let $m\in M$ be a lattice vector such that $\DD(m)$ is an integral divisor, and $m$ and $m+e$ belong to $\relint(\delta_{1,H})$, for any prime divisor $H$.

For every $H\in\supp\DD$, we let $f_H\in A_m$ be a rational function such that $$\ord_H(f_H)=-h_H(m)=-g_{1,H}(m)\,.$$
By our assumption $\varphi f_H\in A_{m+e}$ and so 
$$\ord_H(\varphi f_H)\geq -h_H(m+e)=-g_{1,H}(m+e)\,.$$

This yields $\ord_H(\varphi)\geq -g_{1,H}(m+e)+g_{1,H}(m)=-g_{1,H}(e)$ and so $\varphi\in\Phi_e$. This proves the lemma.
\end{proof}

The following theorem gives a classification of LNDs of fiber type on normal affine $\TT$-varieties. We let $\Phi_e^*=\Phi_e\setminus \{0\}$.

\begin{theorem} \label{fiber}
To any triple $(\rho,e,\varphi)$, where $\rho$ is a ray of $\sigma$, $e\in S_\rho$, and $\varphi\in\Phi_e^*$, we can associate a homogeneous LND $\partial_{\rho,e,\varphi}$ of fiber type on $A=A[Y,\DD]$ of degree $e$ with kernel 
$$\ker\partial_{\rho,e,\varphi}=\bigoplus_{m\in\tau_M}A_m\chi^m\,.$$

Conversely, every non-trivial homogeneous LND $\partial$ of fiber type on $A$ is of the form $\partial=\partial_{\rho,e,\varphi}$ for some ray $\rho\subseteq\sigma$, some lattice vector $e\in S_\rho$, and some function $\varphi\in\Phi_e^*$.
\end{theorem}

\begin{proof}
Letting $\bar{A}=K_Y[\sigma^\vee_M]$, we consider the $K_Y$-LND $\partial_{\rho,e}$ on $\bar{A}$ as in Theorem \ref{toric}. Since $\varphi\in K_Y^*$, $\varphi\partial_{\rho,e}$ is again a $K_Y$-LND on $\bar{A}$.

We claim that $\varphi\partial_{\rho,e}$ stabilizes $A\subseteq \bar{A}$. Indeed, let $f\in A_m\subseteq K_Y$ be a homogeneous element. If $m\in\tau_M$, then $\varphi\partial_{\rho,e}(f\chi^m)=0$. If $m\in\sigma^\vee_M\setminus\tau_M$, then
$$\varphi\partial_{\rho,e}(f\chi^m)=\varphi f\partial_{\rho,e}(\chi^m)=m_0\varphi f\chi^{m+e}\,,$$
where $m_0:=\langle m,\rho_0\rangle\in \ZZ_{>0}$ for the primitive vector $\rho_0$ of the ray $\rho$. By Lemma \ref{phie}, $m_0\varphi f\chi^{m+e}\in A_{m+e}$, proving the claim.

Finally  $\partial_{\rho,e,\varphi}:=\varphi\partial_{\rho,e}|_A$ is a homogeneous LND on $A$ with kernel 
$$\ker\partial_{\rho,e,\varphi}=A\cap\ker\partial_{\rho,e}=\bigoplus_{m\in\tau_M}(A_m\cap K_Y)\chi^m=\bigoplus_{m\in\tau_M}A_m\chi^m\,,$$
as desired.

To prove the converse, we let $\partial$ be a homogeneous LND on $A$ of fiber type. Since $\partial$ is of fiber type, $\partial|_{K_Y}=0$ and so $\partial$ can be extended to a $K_Y$-LND $\bar{\partial}$ on the affine semigroup algebra $\bar{A}=K_Y[\sigma^\vee_M]$. By Theorem \ref{toric} we have $\bar{\partial}=\varphi\partial_{\rho,e}$ for some ray $\rho$ of $\sigma$, some $e\in S_\rho$ and some $\varphi\in K_Y^*$. Since $A$ is stable under $\varphi\partial_{\rho,e}$, by Lemma \ref{phie} $\varphi\in\Phi_e^*$ and so $\partial=\varphi\partial_{\rho,e}|_A=\partial_{\rho,e,\varphi}$.
\end{proof}

\begin{corollary} \label{fgfib}
Let $A$ be a normal finitely generated effectively $M$-graded domain, where $M$ is a lattice of finite rank, and let $\partial$ be a homogeneous LND on $A$. If $\partial$ is of fiber type, then $\ker\partial$ is finitely generated.
\end{corollary}

\begin{proof}
Let $A=A[Y,\DD]$, where $\DD$ is a proper $\sigma$-polyhedral divisor on a semiprojective variety $Y$. In the notation of Theorem \ref{fiber} we have $\partial=\partial_{\rho,e,\varphi}$, where $\rho$ is a ray of $\sigma$. Letting $\tau\subseteq\sigma^\vee$ be the codimension 1 face dual to $\rho$, by Theorem \ref{fiber} we have $\ker\partial=\bigoplus_{m\in\tau_M} A_m\chi^m$.

Let $a_1,\ldots,a_r$ be a set of homogeneous generators of $A$. Without loss of generality, we may assume further that $\deg a_i\in\tau_M$ if and only if $1\leq i\leq s<r$. We claim that $a_1,\ldots,a_s$ generate $\ker\partial$. Indeed, let $P$ be any polynomial such that $P(a_1,\ldots,a_r)\in\ker\partial$. Since $\tau\subseteq\sigma^\vee$ is a face, $\sum m_i\in \tau_M$ for $m_i\in\sigma_M^\vee$ implies that $m_i\in\tau$ $\forall i$. Hence all the monomials composing $P(a_1,\ldots,a_r)$ are monomials in $a_1,\ldots,a_s$, proving the claim.
\end{proof}

\begin{corollary}
Let as before $\partial$ be a homogeneous LND of fiber type on $A=A[Y,\DD]$, and let $f\chi^m\in A\setminus\ker\partial$ be a homogeneous element. Then $\partial$ is completely determined by the image $g\chi^{m+e}:=\partial(f\chi^m)\in A_{m+e}\chi^{m+e}$.
\end{corollary}
\begin{proof}
By the previous theorem $\partial=\partial_{\rho,e,\varphi}$ for some ray $\rho$, some $e\in S_\rho$, and some $\varphi\in\Phi_e$, where $e=\deg\partial$ and $\rho$ is uniquely determined by $e$, see Corollary 2.8 in \cite{Lie08}.

In the course of the proof of Lemma \ref{fiber} it was shown that $\partial_{\rho,e,\varphi}(f\chi^m)=m_0\varphi f\chi^{m+e}$. Thus $\varphi=\frac{g}{m_0f}\in K_0$ is also uniquely determined by our data.
\end{proof}

It might happen that $\Phi_e^*$ as above is empty. Given a ray $\rho\subseteq\sigma$, in the following lemma we give a criterion for the existence of $e\in S_\rho$ such that $\Phi_e^*$ is non-empty.

\begin{theorem} \label{conv-fib}
Let $A=A[Y,\DD]$, and let $\rho\subseteq\sigma$ be the ray dual to a codimension one face $\tau\subseteq \sigma^\vee$. There exists $e\in S_\rho$ such that $\dim\Phi_e$ is positive if and only if the divisor $\DD(m)$ is big for all lattice vectors $m\in\relint(\tau)$.
\end{theorem}
\begin{proof}
Assuming that $\DD(m)$ is big for all lattice vector $m\in\relint(\tau)$, we consider the linear map
$$G:M_{\QQ}\rightarrow \Div_\QQ(Y),\quad m\mapsto\sum_H g_{1,H}(m)\cdot H\,,$$ 
so that $G(m)=\DD(m)$ for all $m\in\tau$ and $D_e=-G(e)$ for all $e\in S_\rho$. Choosing $m\in\relint(\tau)\cap(S_\rho+\mu)$ and $r\in\ZZ_{>0}$, we let
$j=m-\tfrac{1}{r}\cdot\mu$. We consider the divisor
$$G(j)=G(m)-\tfrac{1}{r}\cdot G(\mu)=\DD(m)-\tfrac{1}{r}\cdot G(\mu)\,.$$

Since $\DD(m)$ is big and the big cone is open in $\Div_\RR(Y)$ (see \cite[Def. 2.2.25]{Laz04}), by choosing $r$ big enough, we may assume that $G(j)$ is big. Furthermore, possible increasing $r$, we may assume that $G(r\cdot j)$ has a section. Now, $r\cdot j=r\cdot m-\mu=(r-1)\cdot m+ (m-\mu)$. Since $(r-1)\cdot m\in\tau_M$ and $m-\mu\in S_\rho$, we have $r\cdot j\in S_\rho$. Letting $e=r\cdot j\in S_\rho$ we obtain $D_e=-G(e)$ and so $\dim H^0(Y,O_Y(-D_e))$ is positive.

Assume now that there is $m\in\relint(\tau)$ such that $\DD(m)$ is not big. Since the set of big divisors is and open and convex subset in $\Div_\RR(Y)$, the divisor $\DD(m)$ is not big whatever is $m\in\tau$. We let $B$ be the algebra 
$$B=\bigoplus_{m\in\tau_M}A_m\chi^m\,.$$
Under our assumption $\dim B<n+k-1$. Since $\dim A=n+k$, by Lemma \ref{LND} $(i)$ $B$ cannot be the kernel of an LND on $A$. The latter implies, by Theorem \ref{fiber} that there is no $e\in S_\rho$ such that $\dim\Phi_e$ is positive.
\end{proof}

Finally, we deduce the following corollary.

\begin{corollary} \label{feq}
Two homogeneous LNDs of fiber type $\partial=\partial_{\rho,e,\varphi}$ and $\partial'=\partial_{\rho',e',\varphi'}$ on $A=A[Y,\DD]$ are equivalent if and only if $\rho=\rho'$. Furthermore, the equivalence classes of homogeneous LNDs of fiber type on $A$ are in one to one correspondence with the rays $\rho\subseteq\sigma$ such that $\DD(m)$ is big $\forall m\in\relint(\tau)$, where $\tau$ is the codimension 1 face dual to $\rho$.
\end{corollary}
\begin{proof}
The first assertion follows from the description of $\ker\partial_{\rho,e,\varphi}$ in Lemma \ref{fiber}. The second follows from the first one due to Theorem \ref{conv-fib}.
\end{proof}

\section{Homogeneous Makar Limanov invariant}

Let $X=\spec A$, where $A$ is a finitely generated normal domain, and let $\LND(A)$ be the set of all LNDs on $A$. The \emph{Makar-Limanov invariant} (ML invariant for short) of $A$ (or of $X=\spec A$) is defined as $$\ML(A)=\bigcap_{\partial\in\LND(A)}\ker\partial\,.$$
In the case where $A$ is effectively $M$-graded we let $\LND_{\mathrm{h}}(A)$ be the set of all homogeneous LNDs on $A$ and $\LND_{\mathrm{fib}}(A)$ be the set of all homogeneous LNDs of fiber type on $A$. Following \cite{Lie08}, we define 
$$\ML_{\mathrm{h}}(A)=\bigcap_{\partial\in\LND_{\mathrm{h}}(A)}\ker\partial\qquad \mbox{and} \qquad
\ML_{\mathrm{fib}}(A)=\bigcap_{\partial\in\LND_{\mathrm{fib}}(A)}\ker\partial\,.$$
Clearly $\ML(A)\subseteq\ML_{\mathrm{h}}(A)\subseteq\ML_{\mathrm{fib}}(A)$.

In this section we provide examples showing that, in general, these inclusions are strict and so, the homogeneous LNDs are not enough to compute the ML invariant.

\begin{example}
Let $A=\KK[x,y]$ with the grading given by $\deg x=0$ and $\deg y=1$. In this case, both partial derivatives $\partial_x=\partial/\partial x$ and $\partial_y=\partial/\partial y$ are homogeneous. Since $\ker\partial_x=\KK[y]$ and $\ker\partial_y=\KK[x]$ we have $\ML_{\mathrm{h}}=\KK$. Furthermore, it is easy to see that there is only one equivalence class of LNDs of fiber type. A representative of this class is $\partial_y$ (see Corollary \ref{feq}). This yields $\ML_{\mathrm{fib}}(A)=\KK[x]$. Thus $\ML_{\mathrm{h}}(A)\subsetneq\ML_{\mathrm{fib}}(A)$ in this case.
\end{example}

\begin{example}
To provide an example where $\ML(A)\subsetneq\ML_{\mathrm{h}}(A)$ we consider the Koras-Russell threefold $X=\spec A$, where
$$A=\KK[x,y,z,t]/(x+x^2y+z^2+t^3)\,.$$
The ML invariant was first introduced in \cite{KaMa97} to prove that $X\not\simeq\AF^3$. In fact $\ML(A)=\KK[x]$ while $\ML(\AF^3)=\KK$. In the recent paper \cite{Dub08} Dubouloz shows that the cylinder over the Koras-Russell threefold has trivial ML invariant i.e., $\ML(A[w])=\KK$, where $w$ is a new variable.

Let $A[w]$ be graded by $\deg A=0$ and $\deg w=1$, and let $\partial$ be a homogeneous LND on $A[w]$. If $e:=\deg\partial \leq -1$ then $\partial(A)=0$ and by Lemma \ref{LND} $(i)$ we have that $\ker\partial=A$ and $\partial$ is equivalent to the partial derivative $\partial/\partial w$.

If $e\geq 0$ then $\partial(w)=aw^{e+1}$, where $a\in A$ and so, by Lemma \ref{LND} $(vi)$ $w\in\ker\partial$. Furthermore, for any $a\in A$ we have $\partial(a)=bw^e$, for a unique $b\in A$. We define a derivation $\bar\partial:A\rightarrow A$ by $\bar\partial(a)=b$. Since $\partial^r(a)=\bar\partial^r(a)w^{re}$ the derivation $\bar\partial$ is LND. This yields $\ML_h(A[w])=\ML(A)=\KK[x]$ while $\ML(A[w])=\KK$.
\end{example}

\section{Birational equivalence classes of varieties with \\trivial ML invariant}

In this section we establish the following birational characterization of normal affine varieties with trivial ML invariant. Let $\KK$ be an algebraically closed field of characteristic 0.

\begin{theorem} \label{birML}
Let $X=\spec A$ be an affine variety of dimension $n\geq 2$ over $\KK$. If $\ML(X)=\KK$ then $X \simeq_{\mathrm{bir}} Y\times\PP^2$ for some variety $Y$. Conversely, in any birational class $Y\times\PP^2$ there is an affine variety $X$ with $\ML(X)=\KK$.
\end{theorem}
\begin{proof}
Let $K=\fract A$ be the field of rational functions on $X$ so that $\trdeg_{\KK}(K)=n$. As usual $\trdeg_{\KK}(K)$ denotes the transcendence degree of the field extension $\KK\subseteq K$.
 
Since $\ML(X)=\KK$, there exists at least 2 non-equivalent LNDs $\partial_1,\partial_2:A\rightarrow A$. We let $L_i=\fract(\ker\partial_i)\subseteq K$, for $i=1,2$. By Lemma \ref{LND} $(vii)$, $L_i\subseteq K$ is a purely transcendental extension of degree 1, for $i=1,2$.

We let $L=L_1\cap L_2$. By an inclusion-exclusion argument we have $\trdeg_L(K)=2$. We let $\bar{A}$ be the 2-dimensional algebra over $L$
$$\bar{A}= A\otimes_{\KK} L\,.$$
Since $\fract\bar{A}=\fract A=K$ and $L\subseteq\ker\partial_i$ for $i=1,2$, the LND $\partial_i$ extends to a locally nilpotent $L$-derivation $\bar\partial_i$ by setting
$$\bar{\partial_i}(a\otimes l)=\partial_i(a)\otimes l, \quad \mbox{where}\quad a\in A, \mbox{ and } l\in L\,.$$

Furthermore, $\ker\bar{\partial}_i=\bar{A}\cap L_i$, for $i=1,2$ and so
$$\ker\bar{\partial}_1\cap \ker\bar{\partial}_2=\bar{A}\cap L_1\cap L_2=L\,.$$
Thus the Makar-Limanov invariant of the 2-dimensional $L$-algebra $\bar{A}$ is trivial.

By the theorem in \cite[p. 41]{M-L}, $\bar{A}$ is isomorphic to an $L$-subalgebra of $L[x_1,x_2]$, where $x_1,x_2$ are new variables. Thus
$$K\simeq L(x_1,x_2), \quad\mbox{and so}\quad X\simeq_{\mathrm{bir}}Y\times\PP^2\,,$$
where $Y$ is any variety with $L$ as the field of rational functions. 

The second assertion follows from Lemma \ref{ML0-ex} bellow. This completes the proof.
\end{proof}

The following lemma provides examples of affine varieties with trivial ML invariant in any birational class $Y\times\PP^n$, $n\geq 2$. It is a generalization of Section 4.3 in \cite{Lie08}. Let us introduce some notation.

As before, we let $N$ be a lattice of rank $n\geq 2$ and $M$ be its dual lattice.  We let $\sigma\subseteq N_{\QQ}$ be a pointed polyhedral cone of full dimension. We fix $p\in\relint(\sigma)\cap M$. We let $\Delta=p+\sigma$ and $h=h_{\Delta}$ so that 
$$h(m)=\langle p,m\rangle>0, \quad \mbox{for all } m\in\sigma^\vee\setminus\{0\}\,.$$

Furthermore, letting $Y$ be a projective variety and $H$ be a semiample and big Cartier $\ZZ$-divisor on $Y$, we let $A=A[Y,\DD]$, where $\DD$ is the proper $\sigma$-polyhedral divisor $\DD=\Delta\cdot H$, so that 
$$\DD(m)=\langle p,m\rangle\cdot H,\quad  \mbox{for all } m\in\sigma^\vee\,.$$

Recall that $\fract A=K_Y(M)$ so that $\spec A\simeq_{\mathrm{bir}} Y\times\PP^n$.
\begin{lemma} \label{ML0-ex}
With the above notation, the affine variety $X=\spec A[Y,\DD]$ has trivial ML invariant.
\end{lemma}

\begin{proof}
Let $\{\rho_i\}_i$ be the set of all rays of $\sigma$ and $\{\tau_i\}_i$ the set of the corresponding dual codimension 1 faces of $\sigma^\vee$. Since $rH$ is big for all $r>0$, Theorem \ref{conv-fib} shows that there exists $e_i\in S_{\rho_i}$ such that $\dim\Phi_{e_i}$ is positive, and so we can chose a non-zero $\varphi_i\in\Phi_{e_i}$. In this case, Theorem \ref{fiber} shows that there exists a non-trivial locally nilpotent derivation $\partial_{\rho_i,e_i,\varphi_i}$, with
$$\ker\partial_{\rho_i,e_i,\varphi_i}=\bigoplus_{m\in\tau_i\cap M}A_m\chi^m\,.$$

Since the cone $\sigma$ is pointed and has full dimension, the same holds for $\sigma^\vee$. Thus, the intersection of all codimension 1 faces reduces to one point $\bigcap_i\tau_i=\{0\}$ and so
$$\bigcap_{i}\ker\partial_{\rho_i,e_i,\varphi_i}\subseteq A_0=H^0(Y,\OO_Y)=\KK\,.$$
This yields
$$\ML(A)=\ML_{\mathrm{h}}(A)=\ML_{\mathrm{fib}}(A)=\KK\,.$$
\end{proof}

With the notation as in the proof of Lemma \ref{ML0-ex}, we can provide yet another explicit construction. We fix isomorphisms $M\simeq \ZZ^n$ and $N\simeq \ZZ^n$ such that the standard bases $\{\mu_1,\cdots,\mu_n\}$ and $\{\nu_1,\cdots,\nu_n\}$ for $M_{\QQ}$ and $N_{\QQ}$, respectively, are mutually dual. We let $\sigma$ be the first quadrant in $N_{\QQ}$, and $p=\sum_{i}\nu_i$, so that
$$h(m)=\sum_{i}m_i, \mbox{ and } \DD(m)=\sum_{i}m_i\cdot H, \quad \mbox{where}\quad  m=(m_1,\cdots,m_n), \mbox{ and } m_i\in\QQ_{\geq 0}\,.$$ 

We let $\rho_i\subseteq\sigma$ be the ray spanned by the vector $\nu_i$, and let $\tau_i$ be its dual codimension 1 face. In this setting, $S_{\rho_i}=(\tau_i-\mu_i)\cap M$. Furthermore, letting $e_{i,j}=-\mu_i+\mu_j$ (where $j\neq i$) yields
$$h(m)=h(m+e_{i,j}),\quad \mbox{so that} \quad D_{e_{i,j}}=0, \quad \mbox{and} \quad \Phi_{e_{i,j}}=H^0(Y,\OO_Y)=\KK\,.$$

Choosing $\varphi_{i,j}=1\in\Phi_{e_{i,j}}$ we obtain\footnote{Recall that $\partial_{\nu_i}$ is the partial derivatives defined in Section \ref{LND-toric}.}
$$\partial_{i,j}:=\partial_{\rho_i,e_{i,j},\varphi_{i,j}}=\chi^{\mu_j}\partial_{\nu_i},\quad \mbox{where} \quad i,j\in\{1,\cdots,n\},\ i\neq j\,$$
is a homogeneous LND on $A=A[Y,\DD]$ with degree $e_{i,j}$ and kernel
$$\ker\partial_{i,j}=\bigoplus_{\tau_i\cap M}A_m\chi^m\,.$$
As in the proof of Lemma \ref{ML0-ex} the intersection
$$\bigcap_{i,j}\ker\partial_{i,j}=\KK, \quad \mbox{and so} \quad \ML(X)=\KK\,.$$

We can give a geometrical description of $X$. Consider the $\OO_Y$-algebra
$$\tA=\bigoplus_{m\in\sigma^\vee_M}\OO_Y(\DD(m))\chi^m, \quad \mbox{so that} \quad A=H^0(Y,\tA)\,.$$
In this case, we have
$$\tA=\bigoplus_{r=0}^\infty\,\bigoplus_{\sum m_i=r,\ m_i\geq0}\OO_Y(rH)\chi^m\simeq\Sym\left(\bigoplus_{i=1}^{n}\OO_Y(H)\right)\,.$$
And so $\tX=\SPEC_Y\tA$ is the vector bundle associated to the locally free sheaf $\bigoplus_{i=1}^{n}\OO_Y(H)$ (see Ch. II Ex. 5.18 in \cite{Har77}). We let $\pi:\tX\rightarrow Y$ be the corresponding affine morphism.

The morphism $\varphi:\tX\rightarrow X$ induced by taking global sections corresponds to the contraction of the zero section to a point $\bar{0}$. We let $\theta:=\pi\circ\varphi^{-1}:X\setminus\{\bar{0}\}\rightarrow Y$. The point $\bar{0}$ corresponds to the augmentation ideal $A\setminus \KK$. It is the only attractive fixed point of the $\TT$-action. The orbit closures of the $\TT$-action on $X$ are $\Theta_y:=\overline{\theta^{-1}(y)}=\theta^{-1}(y)\cup \{0\}$, $\forall y\in Y$. Let $\chi^{\mu_i}=u_i$. $\Theta_y$ is equivariantly isomorphic to $\spec\KK[\sigma^\vee_M]=\spec \KK[u_1,\cdots,u_n]\simeq \AF^n$. 

The $\GA$-action $\phi_{i,j}:\GA\times X\rightarrow X$ induced by the homogeneous LND $\partial_{i,j}$ restricts to a $\GA$-action on $\Theta_y$ given by
$$\phi_{i,j}|_{\Theta_Y}:\GA\times\AF^n\rightarrow\AF^n,\quad \mbox{where}\quad u_i\mapsto u_i+tu_j,\quad u_r\mapsto u_r,\ \forall r\neq i\,.$$

Moreover, the unique fixed point $\bar{0}$ is singular unless $Y$ is a projective space and there is no other singular point. By Theorem 2.9 in \cite{Lie09} $X$ has rational singularities if and only if $\OO_Y$ and $\OO_Y(H)$ are acyclic. The latter assumption can be fulfilled by taking, for instance, $Y$ toric or $Y$ a rational surface, and $H$ a large enough multiple of an ample divisor.

\section{FML invariant}

The main application of the ML invariant is to distinguish some varieties from the affine space. Nevertheless, this invariant is far from being optimal as we have seen in the previous section. Indeed, there is a large class of non-rational normal affine varieties with trivial ML invariant. To eliminate such a pathology, we propose below a generalization of the classical ML invariant.

Let $A$ be a finitely generated normal domain. We define the FML invariant of $A$ as the subfield of $K=\fract A$ given by
$$\FML(A)=\bigcap_{\partial\in\LND(A)}\fract(\ker\partial)\,.$$
In the case where $A$ is $M$-graded we define $\FML_{\mathrm{h}}$ and $\FML_{\mathrm{fib}}$ in the analogous way.

\begin{remark}
Let $A=\KK[x_1,\cdots,x_n]$ so that $K=\KK(x_1,\cdots,x_n)$. For the partial derivative $\partial_i=\partial/\partial x_i$ we have $\fract(\ker\partial_i)=\KK(x_1,\cdots,\widehat{x_i},\cdots,x_n)$, where $\widehat{x_i}$ means that $x_i$ is omitted. This yields
$$\FML(A)\subseteq\bigcap_{i=1}^n \fract(\ker\partial_i)=\KK\,,$$
and so $\FML(A)=\KK$. Thus, the FML invariant of the affine space is trivial.
\end{remark}

For any finitely generated normal domain $A$ there is an inclusion $\ML(A)\subseteq\FML(A)$. A priori, since $\FML(\AF^n)=\KK$ the FML invariant is stronger than the classical one in the sense that it can distinguish more varieties form the affine space that the classical one. In the next proposition we show that the classical ML invariant can be recovered from the FML invariant.
\begin{proposition}
Let $A$ be a finitely generated normal domain, then $$\ML(A)=\FML(A)\cap A\,.$$
\end{proposition}
\begin{proof}
We must show that for any LND $\partial$ on $A$,
$$\ker\partial=\fract(\ker\partial)\cap A\,.$$

The inclusion ``$\subseteq$'' is trivial. To prove the converse inclusion, we fix $a\in \fract(\ker\partial)\cap A$. Letting $b,c\in\ker\partial$ be such that $ac=b$, Lemma \ref{LND} $(ii)$ shows that $a\in\ker\partial$.
\end{proof}

Let $A=A[Y,\DD]$ for some proper $\sigma$-polyhedral divisor $\DD$ on a normal semiprojective variety $Y$. In this case $K=\fract A=K_Y(M)$, where $K_Y(M)$ corresponds to the field of fractions of the semigroup algebra $K_Y[M]$. It is a purely transcendental extension of $K_Y$ of degree $\rank M$.

Let $\partial$ be a homogeneous LND of fiber type on $A$. By definition, $K_Y\subseteq \fract(\ker\partial)$ and so, $K_Y\subseteq\FML_{\mathrm{fib}}(A)$. This shows that the pathological examples as in Lemma \ref{ML0-ex} cannot occur. Let us formulate the following conjecture.

\begin{conjecture} \label{conje}
Let $X$ be an affine variety. If $\FML(X)=\KK$ then $X$ is rational.
\end{conjecture}

The following lemma proves Conjecture \ref{conje} in the particular case where $X\simeq_{\mathrm{bir}}C\times\PP^n$, with $C$ a curve.

\begin{lemma} \label{FML-cp1}
Let $X=\spec A$ be an affine variety such that $X\simeq_{\mathrm{bir}}C\times\PP^n$, where $C$ is a curve with field rational functions $L$. If $C$ has positive genus then $\FML(X)\supseteq L$. In particular, if $\FML(X)=\KK$ then $C$ is rational.
\end{lemma}
\begin{proof}
Assume that $C$ has positive genus. We have $K=\fract A=L(x_1,\ldots,x_n)$, where $x_1,\ldots,x_n$ are new variables.

We claim that $L\subseteq\FML(A)$. Indeed, let $\partial$ be an LND on $A$ and let $f,g\in L\setminus \KK$. Since $\trdeg_\KK(L)=1$, there exists a polynomial $P\in\KK[x,y]\setminus\KK$ such that $P(f,g)=0$. Applying the derivation $\partial:K\rightarrow K$ to $P(f,g)$ we obtain
$$\frac{\partial P}{\partial x}(f,g)\cdot \partial(f)+\frac{\partial P}{\partial y}(f,g)\cdot \partial(g)=0\,.$$

Since $f$ and $g$ are not constant we may suppose that $\frac{\partial P}{\partial x}(f,g)\neq 0$ and $\frac{\partial P}{\partial y}(f,g)\neq 0$. Hence $\partial(f)=0$ if and only if $\partial(g)=0$. This shows that one of the two following possibilities occurs:
$$L\subseteq\fract(\ker\partial) \quad \mbox{or} \quad L\cap\fract(\ker\partial)=\KK\,.$$

Assume first that $L\cap\fract(\ker\partial)=\KK$. Then, by Lemma \ref{LND} $(i)$  $\fract(\ker\partial)=\KK(x_1,\ldots,x_n)$ and so the field extension $\fract(\ker\partial)\subseteq K$ is not purely transcendental. This contradits Lemma \ref{LND} $(vii)$. Thus $L\subseteq \fract(\ker\partial)$ proving the claim and the lemma.
\end{proof}

\begin{remark}
We can apply Lemma \ref{FML-cp1} to show that the FML invariant carries more information than usual ML invariant. Indeed, let, in the notation of  Lemma \ref{ML0-ex}, $Y$ be a smooth projective curve of positive genus. Lemma \ref{ML0-ex} shows that $\ML(A[Y,\DD])=\KK$. While by Lemma \ref{FML-cp1}, $\FML(A[Y,\DD])\supseteq K_Y$.
\end{remark}

In the following theorem we prove Conjecture \ref{conje} in dimension at most 3.

\begin{theorem} \label{FML3}
Let $X$ be an affine variety of dimension $\dim X\leq 3$. If $\FML(X)=\KK$ then $X$ is rational.
\end{theorem}
\begin{proof}
Since $\FML(X)$ is trivial, the same holds for $\ML(X)$. If $\dim X\leq 2$ then $\ML(X)=\KK$ implies $X$ rational (see e.g., \cite[p. 41]{M-L}). Assume that $\dim X=3$. Lemma \ref{birML} implies that $X\simeq_{\mathrm{bir}}C\times\PP^2$ for some curve $C$. While by Lemma \ref{FML-cp1}, $C$ is a rational curve.
\end{proof}

\bibliographystyle{alpha}
\bibliography{lnd2}

\end{document}